\documentclass[11pt]{article}
\usepackage{arxiv}
\usepackage{amsfonts}
\usepackage{amsmath}
\usepackage{amsthm}
\usepackage{mathtools}
\usepackage{tocloft}
\usepackage{todonotes}

\usepackage{hyperref}

\theoremstyle{definition}
\newtheorem{definition}{Definition}[]

\newtheorem{setting}[definition]{Setting}
\theoremstyle{plain}
\newtheorem{theorem}[definition]{Theorem}
\newtheorem{corollary}[definition]{Corollary}
\newtheorem{proposition}[definition]{Proposition}
\newtheorem{lemma}[definition]{Lemma}

\newcommand{\norm}[1]{\left\lVert#1\right\rVert}

\begin{document}

\title{Deep Ritz revisited}

\author{
    Johannes M\"uller \\
    Max Planck Institute for Mathematics in the Sciences \\
    \texttt{jmueller@mis.mpg.de}
   \And
    Marius Zeinhofer \\
    University of Freiburg\\
    \texttt{marius.zeinhofer@mathematik.uni-freiburg.de}
}

\maketitle

\begin{abstract}
Recently, progress has been made in the application of neural networks to the numerical analysis of partial differential equations (PDEs) (cf. \cite{weinan2017deep}, \cite{weinan2018deep}). In the latter the variational formulation of the Poisson problem is used in order to obtain an objective function – a regularised Dirichlet energy – that was used for the optimisation of some neural networks. Although this approach showed good visual performance and promising empirical results it is lacking any convergence guarantees. In this notes we use the notion of \(\Gamma\)-convergence to show that ReLU networks of growing architecture that are trained with respect to suitably regularised Dirichlet energies converge to the true solution of the Poisson problem. We discuss how this approach generalises to arbitrary variational problems under certain universality assumptions of neural networks and see that this covers some nonlinear stationary PDEs like the \(p\)-Laplace. 
\end{abstract}

\keywords{Deep ReLU networks, $\Gamma$-Convergence, Poisson Problem, Variational Problems}

\tableofcontents

\section{Introduction}
Artificial neural networks play a key role in current machine learning research and both their performance in practice as well as numerous of their theoretical properties are studied extensively. After the intial success of neural networks in statistical learning problems \cite{krizhevsky2012imagenet} neural networks where used in a variety of fields like generative learning and reinforcement learning but also inverse problems and the numerical analysis of PDEs (cf. \cite{mccann2017convolutional} and \cite{weinan2017deep} respectively). Although partial differential equations are deterministic in nature, some of the deep learning based approaches to their numerical solution are similar to traditional learning problems. In fact, they use the Feynman-Kac formula to approximately evaluate the true solution via Monte-Carlo sampling and then use neural networks for regression of those evaluations (cf. \cite{berner2018analysis}). However, one can also use the variational formulation of PDEs for example for elliptic PDEs and use the associated variational energy as an objective function for the optimisation of the neural networks. In \cite{weinan2018deep} it was demonstrated that this approach yields appealing visual results and a short empirical convergence analysis was carried out. However, it does not provide any guarantees of convergence of this numerical framework for variational problems which is the subject of this work.

\subsection*{Main result}

We will present a version of our main result on the convergence of ReLU networks trained with respect to a regularised Dirichlet energy towards the solution of the Poisson problem and use the following notations. First, we will consider growing sets \(\Theta_n\) of neural networks that we will assume to be universal approximators in \(H^1_0(\Omega)\) and introduce objective functions
\[L_n\colon \Theta_n\to \mathbb R, \quad \theta\mapsto \frac12\int_\Omega \left(\left\lvert\nabla u_\theta \right\rvert^2 - fu_\theta\right)\mathrm dx + n\int_{\partial \Omega} u_\theta^2\mathrm ds, \]
where \(u_\theta\) is the neural network arising from the parameters \(\theta\) and \(f\in L^2(\Omega)\). Now, we assume that \((\theta_n)_{n\in\mathbb N}\) is a sequence of quasi-minimisers of the objective functions, meaning
\[L_n(\theta_n) \le \inf_{\theta\in\Theta_n} L_n + \delta_n, \]
where \(\delta_n\to0\). Then it follows from Theorem \ref{mainresult} that \((u_{\theta_n})_{n\in\mathbb N}\) converges to the true solution of the Poisson problem with right hand side \(f\) weakly in \(H^1(\Omega)\) and thus strongly in \(L^2(\Omega)\). Further, we extend this result to abstract variational problems which for example cover the \(p\)-Laplace operator.

\section{Notation and preliminaries}

\subsection{Preliminaries from neural network theory}

\begin{definition}[Neural network]
Let \(d, m, L\in\mathbb N\). A \emph{standard neural network} of depth \(L\) with \emph{input dimension \(d\)} and \emph{output dimension \(m\)} is a tupel 
\[\theta = \left( (A_1, b_1), \dots, (A_L, b_L)\right)\]
of matrix-vector pairs where \(A_l\in\mathbb R^{N_{l}\times N_{l-1}}\) and \(b_l\in\mathbb R^{N_l}\) and \(N_0 = d, N_L = m\). If \(L = 2\) we call the network \emph{shallow} and \emph{deep} otherwise.
\end{definition}

Every matrix vector pair \((A_l, b_l)\) induces an affine linear transformation that we denote with \(T_l\colon \mathbb R^{N_{l-1}} \to\mathbb R^{N_l}\). We refer to the pairs and also to the affine linear transformations as the \emph{layers} of the network and call the layers \(T_1, \dots, T_{L-1}\) the \emph{hidden layers} and say that \(W(\theta)\coloneqq\max_{l=0, \dots, L} N_l\) is the width of the network. Finally, we denote the set of networks of depth \(L\), width \(W\), input dimension \(d\) and output dimension \(m\) by \(\mathcal N\mathcal N_{L, W, d, m}\). We extend this notation and allow \(W=\infty\) and \(L=\infty\) where this means that the width or depth is finite but arbitrary.

\begin{definition}[Realisation of a neural network, \cite{petersen2018optimal}]
Let \(\theta\) be a neural network of depth \(L\) with input dimension \(d\) and output dimension \(m\) and let \(\rho\colon\mathbb R\to\mathbb R\) be a function. The \emph{realisation} of \(\theta\) with respect to the \emph{activation function \(\rho\)} is the function 
\[R_\rho(\theta)\colon\mathbb R^d\to\mathbb R^m, \quad x\mapsto T_L(\rho(T_{L-1}(\rho(\cdots \rho(T_1(x)))))). \]
If \(f\) is the realisation of a neural network \(\theta\) we say that the network \(\theta\) \emph{represents} the function \(f\).
\end{definition}

The distinction of a neural network and its realisation has proven to be immensely useful in the theory of neural networks and we will see how this can be used to obtain approximation results. However, we want to stress that the realisation map is also crucial for the optimisation of networks, since the objective function is usually the composition of the realisastion map and some objective function on a suitable function space and we refer to \cite{petersen2018topological} and \cite{berner2019degenerate}.

\begin{definition}[ReLU activation function]
The \emph{rectified linear unit} or \emph{ReLU activation function} is defined via
\[\rho\colon\mathbb R\to\mathbb R, \quad x\mapsto \max\left\{ 0, x\right\}.\]
\end{definition}

From now on we will only consider ReLU networks mainly because the realisations of ReLU networks are precisely the piecewise linear functions. This will be very well suited for the study of the Poisson problem where the classical approach relies on finite element functions which are again piecewise linear.

\begin{definition}[Piecewise linear function]\label{PiecewiseLinear}
We say a function \(f \colon\mathbb R^d \to\mathbb R\) is \emph{piecewise linear} if there exists a finite set of closed polyhedra whose union is \(\mathbb R^d\), and \(f\) is affine linear over each polyhedron. The \emph{number of linear regions of} \(f\) is the number of maximal connected subsets of \(\mathbb R^d\) over which \(f\) is affine linear.
\end{definition}

Note every piecewise linear functions is continuous by definition since the polyhedra are closed and cover the whole space \(\mathbb R^d\), and affine functions are continuous. 

\begin{theorem}[Universal representation, \cite{arora2016understanding}]
    Every realisation \(R\colon\mathbb R^d\to\mathbb R^m\) is a piecewise linear function. Conversely, every piecewise linear function \(f\colon\mathbb R^d\to\mathbb R^m\) can be represented by a ReLU network of depth at most \(\lceil \log_2(d+1)\rceil +1\).
\end{theorem}

For the proof of this statement we refer to \cite{arora2016understanding} but shall note that they restrict their studies to \(m=1\). However, the statement above follows easily by building parallelised networks which is straight forward, but for an explicit construction of this parallelisation we refer to \cite{petersen2018optimal}.

One benefit of this result is that it explicitly characterises the range of the realisation map for ReLU networks. A lot of properties of neural networks only depend on their realisations and hence this approach can simplify the study of those. In particular, some approximation capabilities of piecewise linear functions can be transferred directly to ReLU networks. As an example we will present the following universality property of ReLU networks which will be crucial in the proof of the \(\Gamma\)-convergence later on. We will not present the proof since it is the direct consequence of the previous result, the approximation result for piecewise linear functions in the appendix and the density of smooth and compactly supported functions in the Sobolev space \(W^{1, p}_0(\Omega)\). Note, that similar approximation results – for example with non zero boundary values – can be obtained in analogue fashion.

\begin{theorem}[Universal approximation]\label{UniversalApproximation}
Let \(\Omega\subseteq\mathbb R^d\) be a bounded and open set and let \(p\in [1, \infty)\). For all functions \(f\in W^{1, p}_0(\Omega)\) and \(\varepsilon>0\) there is a realisation \(g\in W^{1, p}_0(\Omega)\) of a ReLU network of depth \(\lceil \log_2(d+1)\rceil +1\) such that \[\left\lVert f - g \right\rVert_{W^{1, p}(\Omega)}<\varepsilon.\]
\end{theorem}

\subsection{Reminder on $\Gamma$-convergence}
$\Gamma$-convergence is a notion of convergence for energy functionals that is tightly linked to the convergence of their minimisers and thus extremely well suited to rigorously establish approximation results.
We recall the definition of $\Gamma$-convergence, which we already specialise to the case of reflexive spaces with their associated weak topologies, and show one of its most important features: The convergence of quasi-minimisers of the $\Gamma$-converging sequence to the minimiser of the limit. For further reading we point the reader towards the extensive introduction \cite{dal2012introduction}.
\begin{definition}[\(\Gamma\)-convergence]
    Let \(X\) be a reflexive Banach space, \(F_n, F\colon X\to(-\infty, \infty]\), then \((F_n)_{n\in\mathbb N}\) is \emph{\(\Gamma\)-convergent} to \(F\) if the following properties are satisfied.
    \begin{enumerate}
        \item \emph{Liminf inequality:} For every \(x\in X\) and \((x_n)_{n\in\mathbb N}\) with \(x_n\rightharpoonup x\) we have
            \[F(x) \le \liminf_{n\to\infty} F_n(x_n).\]
        \item \emph{Recovery sequence:} For every \(x\in X\) there is \((x_n)_{n\in\mathbb N}\) with \(x_n\rightharpoonup x\) such that
            \[F(x) = \lim_{n\to\infty} F_n(x_n).\]
    \end{enumerate}
\end{definition}

\begin{proposition}[Limits of quasi-minimisers]\label{limits}
Let \(X\) be a reflexive Banach space and \((F_n)_{n\in\mathbb N}\) be a sequence of functionals that \(\Gamma\)-converges to \(F\). Let \((x_n)_{n\in\mathbb N}\) be a sequence of $\delta_n$-\emph{quasi-minimisers}, i.e.,
\[F_n(x_n) \le \inf F_n + \delta_n \quad \text{for all } n\in\mathbb N,\]
where \(\delta_n\to0\). If \(x_n\rightharpoonup x\), then \(x\) is a minimiser of \(F\).
\end{proposition}
\begin{proof}
By the Liminf inequality we have
\[F(x) \le \liminf_{n\to\infty} F_n(x_n) \le \liminf_{n\to\infty} \big(\inf F_n + \delta_n\big) = \liminf_{n\to\infty} \big(\inf F_n\big) \le \inf F.\]
To see the last inequality, we fix \(y\in X\) and choose a recovery sequence \((y_n)_{n\in\mathbb N}\) and obtain
\[F(y) \ge \limsup_{n\to\infty} F_n(y_n) \ge \limsup_{n\to\infty} \big(\inf F_n\big) \ge \liminf_{n\to\infty} \big(\inf F_n\big).\]
Since \(y\) was arbitrary we can take the infimum of \(y\).
\end{proof}

\begin{definition}[Equi-coercivity]
    Let \(X\) be a reflexive Banach space. The sequence \(F_n\colon X\to(-\infty, \infty]\) is called \emph{equi-coercive} if for all \(r\in\mathbb R\) the set
        \[\bigcup_{n\in\mathbb N}\{F_n\le r\} \]
    is bounded in \(X\) (or equivalently relatively compact wrt the weak topology).
\end{definition}

\begin{corollary}[Convergence of quasi-minimisers]\label{MinimiserCorollary}
    Let \(X\) be a reflexive Banach space and \((F_n)_{n\in\mathbb N}\) be a sequence of equi-coercive functionals that \(\Gamma\)-converges to \(F\) which has a unique minimiser \(x\). Then every sequence \((x_n)_{n\in\mathbb N}\) of quasi-minimisers converges weakly to \(x\).
\end{corollary}
\begin{proof}
Let \((x_n)_{n\in\mathbb N}\) be a sequence of quasi-minimisers. Since we have seen
\(\limsup \big(\inf F_n\big) \le \inf F \)
we obtain
\[(x_n)_{n\in\mathbb N} \subseteq \bigcup_{n\in\mathbb N}\{F_n\le r\}\]
for \(r>\inf F\). Hence, by the equicoercivity \((x_n)_{n\in\mathbb N}\) is relatively weakly compact and since \(X\) is reflexive we are left to show that \(x\) is the only limit point of \((x_n)_{n\in\mathbb N}\). So assume we have \(x_{n_k}\rightharpoonup y\), then \(y\) minimses \(F\) by Proposition \ref{limits} and thus by the uniqueness of the minimser we obtain \(y=x\).
\end{proof}

\section{\(\Gamma\)-convergence for variational energies of deep networks}
In this section we prove our main result, namely the $\Gamma$-convergence for regularised Dirichlet energies on ReLU networks. Furthermore we show that this is in fact very robust and can easily be transferred to more general energies. Finally, as an example of the abstract formulation, we show how it can be applied to the $p$-Dirichlet energy which is associated to a non-linear PDE. 

    But before we proceed with the details we will give a quick overview over this approach to solve PDEs. For this, we consider the Poisson problem which is the prototype of an elliptic PDE and given by
    \begin{align*}
        -\Delta u & = f \quad \text{in } \Omega \\
        u & = 0 \quad \text{on } \partial\Omega.
    \end{align*}
    It is well known that for a function \(u\in H^1_0(\Omega)\) and some right hand side \(f\in H^1_0(\Omega)^\prime\) it is equivalent to be a weak solution of the Poisson problem or to be a solution of the variational problem
    \[ u\in \underset{v\in H^1_0(\Omega)}{\arg\min} \; \frac12\int_{\Omega} \left\lvert \nabla v\right\rvert^2 \mathrm dx - f(v). \]
    The minimised functional on the right hand side is called the \emph{Dirichlet energy} and we refer to \cite{struwe1990variational} for further properties like the existence of unique minimisers.
    By the universal approximation property realisations of ReLU networks with zero boundary values are dense in \(H^1_0(\Omega)\). Hence, one could minimise the Dirichlet energy over all such networks. However, contrary to the use of finite element functions it is not straight forward to enforce zero boundary value conditions for realisations of networks. The solution is to relax the problem and consider the realisations of arbitrary networks but to penalise boundary values. This approach was first taken by \cite{weinan2018deep} where an empirical justification for this method is presented. In the following we give a rigorous proof that the networks that are (quasi) optimised with respect to the regularised Dirichlet energy converge towards the true solution of the Poisson problem if the regularisation is increased and if the networks are universal approximators.

\subsection{\(\Gamma\)-convergence for Dirichlet energy}

    We begin by fixing the notation for this section.
    
    \begin{setting}
    For \(n\in\mathbb N\) we denote by $\Theta_n$ the set of neural networks with input dimension $d\in\mathbb N$, output dimension $1$, depth $\lceil \log_2(d+1) \rceil + 1$ of width at most $n$. Then we look at the realisations of the $\Theta_n$ as a subset in $H^1(\Omega)$ which we call $A_n$ i.e., we set $A_n = R_\rho(\Theta_n)\subset H^1(\Omega)$ where $\Omega\subset\mathbb R^d$ is some open, connected and bounded Lipschitz set and the activation function $\rho$ is ReLU. Note that the functions in $A_n$ are indeed members of $H^1(\Omega)$ as realisations of ReLU networks are of piecewise linear functions. 
    \end{setting}
    
    In the following let $f\in H^1(\Omega)'$, that is $f:H^1(\Omega)\to\mathbb R$ is linear and continuous, and define the sequence of functionals $F_n:H^{1}(\Omega)\to(-\infty,\infty]$
    \begin{align*}
        F_n(u) = \begin{cases}\; \displaystyle \frac12 \int_\Omega |\nabla u|^2\mathrm dx + n\int_{\partial\Omega}u^2\mathrm ds - f(u) &u\in A_n, \\[.4cm] \;\infty   &u\notin A_n. \end{cases}
    \end{align*}
    Further, we define the Dirichlet energy $F:H^1(\Omega)\to(-\infty,\infty]$ with zero boundary values
    \begin{align*}
        F(u) = \begin{cases} \; \displaystyle \frac12 \int_\Omega |\nabla u|^2\mathrm dx - f(u) &u\in H^1_0(\Omega), \\[.4cm] \; \infty   &u\in H^1(\Omega)\setminus H^1_0(\Omega)\end{cases}
    \end{align*}
    and we will show in the remainder of the section that the sequence \((F_n)_{n\in\mathbb N}\) of functionals \(\Gamma\)-converges towards \(F\).
    Note that by the Hahn-Banach theorem every functional in $H^1_0(\Omega)'$ can be extended to some $f\in H^1(\Omega)'$. Having cleared our setting we turn to $\Gamma$-convergence results. 
    \begin{theorem}[\(\Gamma\)-convergence]
    The sequence \((F_n)_{n\in\mathbb N}\) of functionals \(\Gamma\)-converges to the Dirichlet energy \(F\) in the weak topology of \(H^1(\Omega)\).
    \end{theorem}
    \begin{proof}
        We start by checking the liminf inequality. We first assume that $u\notin H^{1,2}_0(\Omega)$. By linearity and continuity of the trace operator it follows that $\operatorname{tr}(u_n)\rightharpoonup\operatorname{tr}(u)$ in $L^2(\partial \Omega)$ and by assumption $\operatorname{tr}(u)\neq 0$. Using the weak semicontinuity of $\norm{\,\cdot\,}^2_{L^2(\partial\Omega)}$ it follows that
        \begin{align*}
            \liminf_{n\to\infty}
            \norm{\operatorname{tr}(u_n)}^2_{L^2(\partial\Omega)}
            \geq
            \norm{\operatorname{tr}(u)}^2_{L^2(\partial\Omega)}
            >0.
        \end{align*}
        and therefore clearly
        \[
            \liminf_{n\to\infty}
            \big(
            n\norm{\operatorname{tr}(u_n)}^2_{L^2(\partial\Omega)}
            \big)
            =
            \infty.
        \]
        Using these facts we establish the liminf inequality in this case 
        \begin{align*}
            \liminf_{n\to\infty}F_n(u_n)
            \geq
            \liminf_{n\to\infty} n\norm{\operatorname{tr}(u_n)}^2_{L^2(\partial\Omega)}
            -
            \lim_{n\to\infty}f(u_n)
            =
            \infty
            =
            F(u).
        \end{align*}
        Now we treat the case $u\in H^{1}_0(\Omega)$. Then using weak lower semi-continuity we find
        \begin{align*}
            \liminf_{n\to\infty}F_n(u_n)
            &\geq
            \liminf_{n\to\infty}\frac12\int_\Omega|\nabla u_n|^2\,\mathrm dx - f(u)
            \\
            &\geq
            \frac12\int_\Omega |\nabla u|^2\,\mathrm dx
            -
            f(u)
            =
            F(u).
        \end{align*}
        We are left to construct a recovery sequence. Assume that $u$ does not have zero boundary conditions. Then just choose $u_n = u$ for all $n$. Otherwise assume $u$ lies in $H^{1}_0(\Omega)$. By Theorem \ref{UniversalApproximation} there is a sequence $(u_n)_{n\in\mathbb N}$ such that $u_n\in A_n\cap H^1_0(\Omega)$ and $u_n\to u$ strongly in $H^1(\Omega)$ and hence it follows \[ F_n(u_n) = \frac12\int_\Omega|\nabla u_n|^2\,\mathrm dx -f(u_n) \to F(u). \]
    \end{proof}

Now we check the requirements of Corollary \ref{MinimiserCorollary} which will yield the the convergence of quasi-minimisers.

    \begin{lemma}
        The functional $F:H^1(\Omega)\to(-\infty, \infty]$ as defined above has a unique minimiser.
        \begin{proof}
        The existence follows by the direct method of the calculus of variations \cite{struwe1990variational} and the uniqueness by the strict convexity of $F$.
        \end{proof}{}
    \end{lemma}{}
    
 Before we turn towards the equi-coercity we define the functional \[G_n(u) = \frac12 \int_\Omega |\nabla u|^2\,\mathrm dx + n\int_{\partial\Omega}u^2\,ds - f(u)\] hence $G_n$ agrees with $F_n$ where $F_n\neq \infty$. Note that it trivially holds
    \[
        F_n 
        =
        G_n + \chi_{A_n}
    \]
    where $\chi_{A_n}=0$ on $A_n$ and $\chi_{A_n}=\infty$ otherwise. We also remark that
    \[
        F_n \geq G_n \geq G_1.
    \]
    This means to show the equi-coercivity of $(F_n)_{n\in\mathbb N}$ it suffices to prove the coercity of the single functional $G_1$.
The following lemma uses a classical compactness argument to establish a Poincaré type inequality (cf. \cite{alt1992linear}).
\begin{lemma}
    Let $r>0$ be fixed and consider the set $M \coloneqq \{ G_1\leq r \}\subset H^{1}(\Omega)$ i.e., the functions in $H^{1}(\Omega)$  that satisfy
    \begin{align}\label{InequalityM}
        \frac12\int_\Omega |\nabla u|^2\,\mathrm dx
        +
        \int_{\partial\Omega}u^2\,ds
        -
        f(u)
        \leq
        r.
    \end{align}
    Those functions satisfy a Poincar\'e type inequality of the form
    \begin{align}\label{Poincare}
        \norm{u}_{L^2}
        \leq
        C(\norm{\nabla u}_{L^2} + 1),
    \end{align}
    where \(C\) only depends on \(r\) and \(\Omega\).
    \begin{proof}
        The proof consists of two steps. First we will show that the inequality (\ref{InequalityM}) implies that $M$ cannot contain arbitrarily large, constant functions and second we prove that a failure of the Poincar\'e inequality (\ref{Poincare}) means that $M$ contains any large, constant function hence the assertion follows.
        \ \\ \ \\
        \textbf{Step 1.} Let $\xi\in\mathbb{R}$ be a constant function in $M$. We will show that there is some fixed $C>0$ depending only on $r,\norm{f}_{H^1(\Omega)'}$ and $|\partial\Omega|$ such that 
        \[
            |\xi|\leq C.
        \] 
        Using a scaled version of Young's inequality with $\varepsilon |\Omega|^{1/2} \leq |\partial\Omega|/2$ we compute
        \begin{align*}
            r
            \geq
            \int_{\partial\Omega}\xi^2\,ds
            -
            f(\xi)
            &\geq
            |\xi|^2|\partial\Omega|
            -
            \norm{f}_{H^1(\Omega)'}\norm{\xi}_{H^1(\Omega)}
            \\
            &=
            |\xi|^2|\partial\Omega| 
            -
            \norm{f}_{H^1(\Omega)'}|\Omega|^{1/2}|\xi|
            \\
            &\geq
            |\xi|^2|\partial\Omega|
            -
            C(\varepsilon)\norm{f}^2_{H^1(\Omega)'}
            -\varepsilon|\Omega|\cdot|\xi|^2
            \\
            &\geq
            \frac12|\partial\Omega|\,|\xi|^2
            -
            C(\varepsilon)\norm{f}^2_{H^1(\Omega)'}.
        \end{align*}
        Thus we can solve for $|\xi|$ and find a uniform bound in terms of $r,\norm{f}_{H^1(\Omega)'}$ and $|\partial\Omega|$.
        \ \\ \ \\ 
        \textbf{Step 2.} Now we assume that the inequality fails and will show that this implies that $M$ contains arbitrarily large constant functions. Assume therefore that there is a sequence $(u_k)\subset M$ such that
        \[
            \norm{\nabla u_k}_{L^2} + 1
            \leq
            \frac1k\norm{u_k}_{L^2}.
        \]
        This inequality implies that $\norm{u_k}_{L^2}\to\infty$ and hence for every large but fixed $R>0$ we may assume that $\norm{u_k}_{L^2}^{-1} R\leq 1$ and set $v_k = u_k (R\norm{u_k}_{L^2}^{-1})$. By the star shape of $M$ the $v_k$ are a sequence in $M$ and the above inequality yields upon multiplying 
        \begin{align}\label{GradientInequality}
            \norm{\nabla v_k}_{L^2}
            +
            \frac{R}{\norm{u_k}_{L^2}}
            \leq
            \frac{R}{k}
            \to 0.
        \end{align}
        As $\norm{v_k}=R$ and $(\ref{GradientInequality})$ implies a bound on $\norm{\nabla v_k}_{L^2}$ we extract a weakly $H^{1,2}(\Omega)$ convergent subsequence $v_j\rightharpoonup v$ with limit $v$ in $M$ by the weak closedness of $M$. Also from (\ref{GradientInequality}) we deduce that
        \[
            \nabla v_j \rightharpoonup \nabla v = 0 
            \quad\text{weakly in }L^2(\Omega)^n
        \]
        and thus there is a constant $\xi\in\mathbb{R}$ such that $v(x)=\xi$ up to a set of measure zero in $\Omega$. To identify $\xi$ we employ the Rellich compactness theorem which yields that $v_j\to v$ strongly in $L^2(\Omega)$ and together with $\norm{v_j}_{L^2}=R$ we conclude
        \[
            R 
            = 
            \norm{v}_{L^2} 
            =
            \norm{\xi}_{L^2} 
            =
            |\Omega|^{1/2}|\xi|
        \]
        and as $R>0$ was arbitrary this shows that $M$ contains any large, constant function which manifests the desired contradiction.
    \end{proof}
\end{lemma}
\begin{proposition}
    The sequence $(F_n)$ is equicoercive i.e., for every $r\in\mathbb{R}$ there is a bounded set $K_r\subset H^{1,2}(\Omega)$ such that
    \[
        \bigcup_{n\in\mathbb{N}}\{ F_n \leq r \}
        \subset
        K_r.
    \]
\end{proposition}
    \begin{proof}
    As we discussed before the statement of the theorem it is enough to show that $M=\{ G_1\leq r \}\subset K_r$ for some bounded set $K_r$ in $H^{1}(\Omega)$. The next observation is that upon adding $\norm{\nabla u}_{L^2}$ to the inequality (\ref{Poincare}) we find that 
    \begin{align}\label{ModifiedPoincare}
        \norm{u}_{H^{1}} \leq C(\norm{\nabla u}_{L^2} + 1)
    \end{align}
    hence it is enough to provide an $L^2$ bound on the gradients of the functions in $M$. We employ a scaled version of Young's inequality with fitting $\varepsilon>0$ and compute using the inequality (\ref{ModifiedPoincare})
    \begin{align*}
        r &\geq 
        \frac12\norm{\nabla u}^2_{L^2} 
        - 
        \norm{f}_{H^1(\Omega)'}\norm{u}_{H^{1,2}}
        \\
        &\geq
         \frac12\norm{\nabla u}^2_{L^2} 
        - 
        C\norm{f}_{H^1(\Omega)'}\norm{\nabla u}_{L^2}
        -
        C\norm{f}_{H^1(\Omega)'}
        \\
        &\geq
        \frac12\norm{\nabla u}^2_{L^2} 
        - 
        \varepsilon \norm{\nabla u}_{L^2}^2
        -
        (C(\varepsilon)+C)\norm{f}_{H^1(\Omega)'}
        \\
        &\geq
        \frac14\norm{\nabla u}^2_{L^2}
        -
        (C(\varepsilon)+C)\norm{f}_{H^1(\Omega)'}.
    \end{align*}
    Rearranging the inequality we see that $\norm{\nabla u}_{L^2}\leq C(r)$ which is sufficient to conclude the proof due to (\ref{ModifiedPoincare}).
    \end{proof}

\begin{corollary}
    Any sequence \((u_n)_{n\in\mathbb N}\) of 
    quasi-minimizers 
    of \((F_n)_{n\in\mathbb N}\) converges weakly in \(H^{1}(\Omega)\) and hence strongly in \(L^2(\Omega)\) to the solution $u$ of 
    \begin{align*}
        -\Delta u & = f \quad \text{in } \Omega \\
        u & = 0 \quad \text{on } \partial\Omega.
    \end{align*}
\end{corollary}

In practice, one will not optimise the functionals \(F_n\) but rather work in the parameter space \(\Theta_n\) and use \(L_n\coloneqq F_n\circ R_\rho\) as an objective function for the optimisation. However, if \((\theta_n)_{n\in\mathbb N}\) is a sequence of quasi-minimisers of \((L_n)_{n\in\mathbb N}\), then \((R_\rho(\theta_n))_{n\in\mathbb N}\) is a sequence of quasi-minimisers of \((F_n)_{n\in\mathbb N}\) since \(\{F_n<\infty\} = R_\rho(\Theta_n)\) by definition. Hence, the previous result yields the convergence of \((R_\rho(\theta_n))_{n\in\mathbb N}\) towards the solution of the Poisson problem and we will fix this in the following.

\begin{theorem}\label{mainresult}
Let \(L_n\coloneqq F_n\circ R_\rho\) and \((\theta_n)_{n\in\mathbb N}\) be a sequence of quasi-minimisers of \((L_n)_{n\in\mathbb N}\). Then \((R_\rho(\theta_n))_{n\in\mathbb N}\) converges to the true solution of the Poisson problem, both weakly in \(H^1(\Omega)\) and strongly in \(L^2(\Omega)\).
\end{theorem}

\subsection{Generalisation to abstract energy functionals}
The above considerations admit a direct extension to a considerably broader class of energies, including energies associated to higer order elliptic equations and also non-linear ones such as the $p$-Dirichlet energy, the Euler-Lagrange equations of which is the $p$-Laplace (cf. \cite{struwe1990variational}). In this section the activation function $\rho:\mathbb{R}\to\mathbb{R}$ is no longer assumed to be ReLU but needs to be chosen according to the energy.
\begin{setting}\label{AbstractSetting}
We begin with the assumptions and the norm structure on the space where the energy will live. Assume $X$  is a reflexive Banach space with norm $\norm{\,\cdot\,}_X$ and that there is an additional norm $|\cdot|$ on $X$, which does not need to render $X$ complete. Furthermore let $Y$ be another Banach space with norm $\norm{\,\cdot\,}_Y$ and let 
\( T:(X,\norm{\cdot}_X)\to (Y,\norm{\cdot}_Y) \) be a linear and continuous operator which is linked to the norm $\norm{\,\cdot\,}_X$ in the following way
\[ \norm{x}_X = |x| + \norm{Tx}_Y \quad\text{for all }x\in X. \]
As for the activation function the only thing we require for the moment is that $\bigcup_{n\in\mathbb N}R_\rho(\Theta_n)\subset X$ and we set again $A_n\coloneqq R_\rho(\Theta_n)$. Of course this implies that our Banach space $X$ is some kind of function space.
We turn to the energy and some abstract analogue of boundary values. We assume that there are two maps
\[ E:X\to\mathbb{R} \quad\text{and}\quad \gamma:X\to B \]
where $E$ (the energy) is bounded from below, weakly lower semi-continuous and norm-continuous. Both, the weak topology related to the weak lower semi-continuity and the norm-continuity are meant with respect to the norm $\norm{\cdot}_X$. Furthermore $\gamma$ (trace operator) is a linear and continuous map from $(X,\norm{\cdot})$ into the Banach space $B$ that is the abstract analogue of boundary values. We set $X_0 = \ker(\gamma)$. With this terminology fixed we are able to define our functionals $F_n, F:X\to(-\infty,\infty]$. Let $f\in X'$ and set
    \begin{align*}
        F_n(x)
        =
        \begin{cases}
	    \; E(x) + n\norm{\gamma(x)}^2_B - f(x) 
	    &x\in A_n,
	    \\[.2cm]
	    \; \infty   &x\notin A_n.
	    \end{cases}
    \end{align*}
    and for the limit functional
    \begin{align*}
        F(x)
        =
        \begin{cases}
	    \; E(x) - f(x) 
	    &x\in X_0,
	    \\[.2cm]
	    \; \infty   &x\notin X_0.
	    \end{cases}
    \end{align*}
\end{setting}
To get an intuition for the setting think of $X=H^1(\Omega)$ where $|\cdot|=\norm{\,\cdot\,}_{L^2}$ and $\norm{\,\cdot\,}_X=\norm{\,\cdot\,}_{H^1}$ such as $T:H^1(\Omega)\to L^2(\Omega)^N$ with $u\mapsto \nabla u$.
The question we ask now is: 
\begin{gather*}
    \text{\emph{Under which assumptions do we obtain the $\Gamma$-convergence of \((F_n)_{n\in\mathbb N}\) to \(F\) 
    in the weak topology of $X$?}}
\end{gather*}{}
To this end we formulate the hypothesis below.
\begin{enumerate}
    \item[(H1)]\label{AssumptionDensity} The set $\bigcup_{n\in\mathbb N}R_\rho(\Theta_n)\cap X_0$ is dense in $X_0$ with respect to the norm topology on $(X_0,\norm{\,\cdot\,}_X)$.
    \item [(H2)] The space $X$ is reflexive, its norm is given as $\norm{\,\cdot\,}_X = |\cdot| + \norm{T\,\cdot\,}_Y$, for some other norm $|\cdot|$ on $X$ and $T:(X,\norm{\cdot}_X)\to (Y,\norm{\cdot}_Y)$ is linear and continuous into the Banach space $Y$.
    \item[(H3)] The identity $(X,\norm{\,\cdot\,}_X\to(X,|\cdot|)$ maps weakly convergent sequences to strongly convergent ones.
    \item[(H4)] The map $\gamma$ is linear and continuous and for every $r\in\mathbb{R}$ the set $\{ \norm{\gamma(x)}^2_B-f(x) \leq r \}\cap \ker(T)$ is bounded in $X$ by a constant $C=C(r)<\infty$.
    \item[(H5)] The energy $E:X\to\mathbb{R}$ is bounded form below, weakly lower semi-continuous (with respect to the weak topology induced by $\norm{\cdot}_X$) and also $\norm{\cdot}_X$-continuous and satisfies $|E(x)|\geq c_1\norm{Tx}_Y^p-c_2$ for $p>1$ and constants $c_1,c_2$. Furthermore assume that $E$ has a unique minimiser on $X_0$.
\end{enumerate}
We can formulate our main result, its proof is very similar to the case of the Dirichlet energy 
\begin{theorem}
    Under the hypotheses \textup{(H1)}-\textup{(H5)} the sequence \((F_n)_{n\in\mathbb N}\) of functionals \(\Gamma\)-converges towards \(F\) and further every sequence of quasi-minimisers of \((F_n)_{n\in\mathbb N}\) converges to the unique minimiser of $F$.
    \begin{proof}
        The proof mainly consists of abstraction of the concepts we already met in the Dirichlet case and therefore we will keep it brief. A look back to the shows that (H1), together with the strong continuity of $E$ from (H5) provides a recovery sequence and the $\liminf$ inequality follows by the lower semi-continuity assumed in (H5).
        The hypothesis (H2)-(H4) take care of the equi-coercivity which deserves some extra comments and we refer the reader to the next lemma. 
    \end{proof}
\end{theorem}
\begin{lemma}[Abstract Poincar\'e Inequality]
    Consider the setting described above i.e., let $(X,\norm{\cdot}_X)$ and $(Y,\norm{\cdot}_Y)$ be Banach spaces with a linear and continuous map $T:X\to Y$ and another norm $\lvert\,\cdot\,\rvert$ on $X$ such that $\norm{x}_X=|x|+\norm{Tx}_Y$ and the identity $(X,\norm{\cdot}_X)\to(X,\lvert\,\cdot\,\rvert)$ maps weakly to strongly convergent sequences. Let $M$ be some weakly closed, star-shaped set with center zero. Then $\ker(T)\cap M$ is bounded if and only if there is a constant $C$ such that
    \[ |x|\leq C(\norm{Tx}_Y+1)\quad\textup{for all }x\in M. \]
    \begin{proof}
    This works exactly as in the case $X=H^1(\Omega)$ and $Y=L^2(\Omega)^d$ with $T=\nabla$ which we studied before.
    \end{proof}
\end{lemma}{}
\subsubsection*{The \(p\)-Laplace: An example for a nonlinear PDE}
We illustrate the abstract setting briefly by considering the $p$-Dirichlet energy for $p\in (1,\infty)$ which is given by
\begin{align*}
    E:W^{1,p}(\Omega)\to\mathbb{R}
    \quad\text{with}\quad
    E(u) 
    =
    \frac1p\int_\Omega|\nabla u|^p\,dx
\end{align*}
Note that for $p\neq 2$ the associated Euler-Lagrange equation is nonlinear. This PDE is called the $p$-Laplace equation and, in strong formulation, is given by 
\begin{align*}
    -\operatorname{div}(\lvert\nabla u\rvert^{p-2}\nabla u) &= f \quad\text{in }\Omega
    \\
    u &= 0 \quad\text{on }\partial\Omega,
\end{align*}
compare also to \cite[First Chapter]{struwe1990variational}. Choosing the ReLU activation function, the abstract setting presented in (\ref{AbstractSetting}) is applicable in this case by the following choices
\[ X = W^{1,p}(\Omega), \quad Y = L^p(\Omega)^d, \quad B = L^p(\partial\Omega), \quad |u| = \norm{u}_{L^p(\Omega)}\]
as well as
\begin{align*}
    \gamma = \operatorname{tr}:W^{1,p}(\Omega) & \to L^p(\partial\Omega)
    \\
    T:W^{1,p}(\Omega) & \to L^p(\Omega)^d\quad\text{with}\quad u\mapsto\nabla u.
\end{align*}
Clearly (H1), (H2) and (H5) are fulfilled and (H3) is due to Rellich's embedding theorem. We look at (H4) and need to guarantee that for every $f\in W^{1,p}(\Omega)'$ and $r>0$ the following set is bounded in $W^{1,p}(\Omega)$
\begin{align*}
    \left\{
    \norm{\operatorname{tr}(u)}^2_{L^p(\partial\Omega)}-f(u) \leq r
    \right\}.
\end{align*}{}
This works similar to the case $p=2$ discussed above and uses again a scaled version of Young's inequality. 
\section{Conclusions and further research}

We established a $\Gamma$-convergence result which ensures that ReLU networks can be optimised to approximately solve a large class of stationary PDEs that are Euler-Lagrange equations of variational energies. This can be seen as the first step into the theoretical justification of the deep ritz method introduced in \cite{weinan2018deep}. However, the result is to be enjoyed with a grain of salt and we want to touch on a few limitations. 
\begin{enumerate}
    \item It does not take into account the numerical approximation of the Dirichlet energy which is a further potential source of error. Further, the resulting numerical optimisation problem is not studied.
    \item It does not provide any convergence rates and is of purely qualitative nature, hence it can not explain why NNs should be superior to traditional methods.
\end{enumerate}

For future research we propose the following questions:
\begin{enumerate}
    \item Treat the \(p\)-Laplace with this approach.
    \item Work towards higher order problems; in particular one has to abandon the ReLU activation function since ReLU networks are no longer contained in the considered function spaces.
    \item Show \(\Gamma\)-convergence for numerical approximations of the regularised Dirichlet energies; this would show that – up to optimisation of the networks – the deep Ritz method is consistent.
    \item Establish convergence rates.
\end{enumerate}

\bibliographystyle{apalike}
\bibliography{deep-ritz}

\appendix
\section*{Appendix}
\begin{lemma}
    Let $\varphi\in C_c^\infty(\mathbb{R}^d)$ be a smooth function with compact support. Then for every $\varepsilon>0$ there is a piecewise linear function $s_\varepsilon$ (see Definition \ref{PiecewiseLinear}) such that for all $p\in[1,\infty]$ it holds
    \begin{align*}
        \norm{s_\varepsilon-\varphi}_{W^{1,p}(\mathbb R^d)}<\varepsilon
        \quad \text{and}\quad
        \operatorname{supp}(s_\varepsilon)
        \subset
        \operatorname{supp}(\varphi) + B_{\varepsilon}(0).
    \end{align*}
    \begin{proof}
        In the following we will denote by $\norm{\cdot}_\infty$ the uniform norm on $\mathbb R^d$. To show the assertion choose a triangulation $\mathcal{T}$ of $\mathbb{R}^d$ of width $\delta=\delta(\varepsilon) >0$, consisting of rotations and translations of one non-degenerate simplex $K$. We choose $s_\varepsilon$ to agree with $\varphi$ on all vertices of elements in $\mathcal{T}$. Since $\varphi$ is compactly supported it is uniformly continuous and hence it is clear that $\norm{\varphi-s_\varepsilon}_\infty <\varepsilon$ if $\delta$ is chosen small enough.
        \ \\
        To prove convergence of the gradients we show that also $\norm{\nabla \varphi - \nabla s_\varepsilon}_{\infty}<\varepsilon$ which will be shown on one element $K\in\mathcal{T}$ and as the estimate is independent of $K$ is understood to hold on all of $\mathbb{R}^d$.
        So let $K\in\mathcal{T}$ be given and denote its vertices by $x_1,\dots,x_{n+1}$. We set $v_i = x_{i+1}-x_1$, $i=1,\dots,d$ to be the vectors spanning $K$. By the one dimensional mean value theorem we find $\xi_i$ on the line segment joining $x_1$ and $x_i$ such that
        \[
            \partial_{v_i}s_\varepsilon(v_1) = \partial_{v_i}\varphi(\xi_i).
        \]
        Note that $\partial_{v_i}s_\varepsilon$ is constant on all of $K$ where it is defined. Now for arbitrary $x\in K$ we compute with setting $w=\sum_{i=1}^d\alpha_iv_i$ for $w\in\mathbb{R}^d$ with $|w|\leq 1$. Note that the $\alpha_i$ are bounded uniformly in $w$, where we use that all elements are the same up to rotations and translations.
        \begin{align*}
            \norm{\nabla\varphi(x) - \nabla s_\varepsilon(x)}
            &=
            \sup_{|w|\leq1}\norm{\nabla\varphi(x)w-\nabla s_\varepsilon(x)w}
            \\
            &\leq
            \sup_{|w|\leq 1} \sum_{i=1}^d |\alpha_i|
            \underbrace{
            \norm{\partial_{v_i}\varphi(x) - \partial_{v_i}s_\varepsilon(x) } 
            }_{=(*)}
        \end{align*}
        where again $(*)$ is uniformly small due to the uniform continuity of $\nabla\varphi$. This concludes the proof.
    \end{proof}
\end{lemma}
\end{document}